\documentclass[11pt,reqno]{amsart}
\usepackage{amsmath,amssymb,amsthm}
\usepackage[numbers]{natbib}

\newtheorem{theorem}{Theorem}
\newtheorem{corollary}{Corollary}
\newtheorem{lemma}{Lemma}
\newtheorem{proposition}{Proposition}
\theoremstyle{definition}
\newtheorem{definition}{Definition}
\newtheorem{remark}{Comment}
\newtheorem{example}{Example}

\newcommand{\eps}{\varepsilon}
\newcommand{\LL}{ \mathcal{L}}
\renewcommand{\hat}{\widehat}

\newcommand{\Ep}{{\mathrm{E}}}

\renewcommand{\Pr}{{\mathrm{P}}}

\newcommand{\RR}{\mathbb{R}}

\DeclareMathOperator{\Corr}{Corr}
\DeclareMathOperator{\Var}{Var}

\begin{document}

\title[Comparison and Anti-Concentration]{Comparison and Anti-Concentration Bounds for Maxima of Gaussian Random Vectors}
\author[Chernozhukov]{Victor Chernozhukov}
\author[Chetverikov]{Denis Chetverikov}
\author[Kato]{Kengo Kato}

\address[V. Chernozhukov]{
Department of Economics  \& Operations Research Center, MIT, 50 Memorial Drive, Cambridge, MA 02142, USA.
}
\email{vchern@mit.edu}

\address[D. Chetverikov]{
Department of Economics, UCLA, Bunche Hall, 8283, 315 Portola Plaza, Los Angeles, CA 90095, USA.
}
\email{chetverikov@econ.ucla.edu}

\address[K. Kato]{
Graduate School of Economics, University of Tokyo, 7-3-1 Hongo, Bunkyo-ku, Tokyo 113-0033, Japan.
}
\email{kkato@e.u-tokyo.ac.jp}

\date{First version: January 21, 2013. This version: \today. }

\begin{abstract}
Slepian and Sudakov-Fernique type inequalities, which compare
expectations of maxima of Gaussian random vectors under certain restrictions on the covariance matrices, play an important role in probability
theory, especially in empirical process and extreme value theories.  Here we give
explicit comparisons of expectations of smooth functions and distribution functions of maxima
of Gaussian random vectors without any restriction on the covariance matrices.
We also establish  an anti-concentration inequality
for the maximum of a Gaussian random vector, which derives a useful upper bound
on the L\'{e}vy concentration function for the Gaussian maximum. The bound is dimension-free
and applies to vectors with arbitrary covariance matrices.  This anti-concentration inequality plays a crucial role in establishing bounds on the Kolmogorov distance between maxima of Gaussian random vectors.
These results have immediate applications in mathematical statistics. As an example of application, we establish a conditional  multiplier central limit theorem for maxima of sums of independent random vectors where the dimension of the vectors is  possibly much larger than the sample size.
\end{abstract}

\maketitle

\maketitle

\section{Introduction}

We derive a bound on the difference in expectations of smooth functions of maxima of finite dimensional Gaussian random vectors.
We also derive a bound on  the Kolmogorov distance between distributions of
these maxima.   The key property of these bounds is that they depend on the dimension $p$ of Gaussian random vectors only through $\log p$, and on the max norm of the difference between the covariance matrices of the vectors.  These results  extend and complement the work of \cite{Chatterjee2005b} that derived an explicit Sudakov-Fernique type bound on the difference of expectations of maxima of Gaussian random vectors. See also \cite{AdlerTaylor2007}, Chapter 2.  As an application, we establish a conditional  multiplier central limit theorem for maxima of sums of independent random vectors where the dimension of the vectors is  possibly much larger than the sample size. In all these results, we allow for arbitrary covariance structures between the coordinates in random vectors, which is plausible especially in applications to high-dimensional statistics. We stress that  the derivation of bounds on the Kolmogorov distance is by no means trivial and  relies on a new {\em anti-concentration} inequality for maxima of Gaussian random vectors,
which is  another main result of  this paper (see Comment \ref{rem: concentration vs anticoncentration} for what anti-concentration inequalities here precisely refer to and how they differ from the concentration inequalities).
These anti-concentration bounds are non-trivial in the following sense: (i) they apply to every dimension $p$ and they are {\em dimension-free} in the sense that the bounds depend on the dimension $p$ only through the expectation of the maximum of the Gaussian random vector, thereby admitting direct extensions to the infinite dimensional case, namely, separable Gaussian processes (see \cite{CCK13} for this extension and applications to empirical processes).  This dimension-free nature is parallel to the Gaussian concentration inequality, which states that the supremum concentrates around the expected supremum. (ii) They allow for arbitrary covariance structures between the coordinates in Gaussian random vectors, and (iii) they are sharp in the sense that there is an example for which the bound is tight up to a dimension independent constant. We note that these anti-concentration bounds are sharper than those that result from application of the universal reverse isoperimetric inequality of \cite{Ball1993} (see also \cite{Bentkus2003}, p.386-367).

Comparison inequalities for Gaussian random vectors play an important role in probability theory, especially in empirical process and extreme value theories.
We refer the reader to \cite{Slepian1962}, \cite{Leadbetter1983}, \cite{G85}, \cite{LT91}, \cite{LS01}, \cite{LS02}, \cite{Chatterjee2005b}, and \cite{Y09} for standard references on this topic.
The anti-concentration phenomenon has attracted considerable interest in the context of random matrix theory and the Littlewood-Offord problem in number theory. See, for example, \cite{RV08}, \cite{RV09}, and \cite{VR07} who remarked that \textit{``concentration is better understood than anti-concentration''}.
Those papers were concerned with the anti-concentration in the Euclidean norm for sums of independent random vectors, and  the topic and the proof technique here are substantially different from theirs.

Either of the comparison or anti-concentration bounds derived in the paper have many immediate statistical applications, especially in the context of high-dimensional statistical inference, where the dimension $p$ of vectors of interest is much larger than the sample size (see \cite{BV11} for a textbook treatment of the recent developments of high-dimensional statistics). In particular,
the results established here are helpful in deriving an invariance principle for sums of high-dimensional random vectors, and also in establishing
the validity of the multiplier bootstrap for inference in practice.  We refer the reader to our companion paper \cite{CCK12}, where
the results established here are applied in several important statistical problems, particularly the analysis of Dantzig selector of \cite{CandesTao2007} in the non-Gaussian setting.

 The proof strategy for our anti-concentration inequalities is to directly bound the density function of the maximum of a Gaussian random vector. The paper by \cite{NV09} is concerned with bounding such a density (see \cite{NV09}, Proposition 3.12) but under  positive covariances restriction.  This is related to but different from our anti-concentration bounds. The crucial assumption in their Proposition 3.12 is   positivity of all the covariances between the coordinates in the Gaussian random vector, which does not hold in our targeted applications in high-dimensional statistics, for example, analysis of Danzig selector. Moreover, their upper bound on the density depends on the inverse of the lower bound on the covariances -- and hence, for example, if there are two independent coordinates in the Gaussian random vector, then the upper bound becomes infinite. Our anti-concentration bounds do not require such positivity (or other) assumptions on covariances and hence are not implied by the results of \cite{NV09}.   Another method for deriving reverse isoperimetric inequalities is to use geometric results of  \cite{Nazarov03}, as shown by \cite{Klivans08}, which leads to dimension-dependent anti-concentration inequalities, which are essentially different from ours.   Moreover, our density-bounding proof technique is substantially different from that of \cite{NV09} based on Malliavin calculus or  \cite{Nazarov03} based on geometric arguments.


The rest of the paper is organized as follows. In Section \ref{sec: Gaus and multiplier}, we present comparison bounds for Gaussian random vectors and its application, namely the conditional multiplier central limit theorem. In Section \ref{sec: anti-concentration}, we present anti-concentration bounds for maxima of Gaussian random vectors. In Sections \ref{sec: proof for section 2} and \ref{sec: proof for section 3}, we give proofs of the theorems in Sections \ref{sec: Gaus and multiplier} and \ref{sec: anti-concentration}. The Appendix contains a proof of a technical lemma.

{\bf Notation}.
Denote by $(\Omega,\mathcal{F},\Pr)$  the underlying probability space.
For $a,b \in \RR$, we write $a_{+} = \max \{ 0,a \}$ and $a \vee b = \max \{ a,b \}$. Let $1(\cdot)$ denote the indicator function.
The transpose of a vector $z$ is denoted by $z^{T}$. For a function $g: \RR \to \RR$, we use the notation $\| g \|_{\infty} = \sup_{z \in \RR} | g(z) |$. Let $\phi(\cdot)$ and $\Phi (\cdot)$ denote the density and distribution functions of the standard Gaussian distribution, respectively: $\phi(x) = (1/\sqrt{2 \pi}) e^{-x^{2}/2}$ and $\Phi (x) = \int_{-\infty}^{x} \phi(t) dt$.

\section{Comparison Bounds and Multiplier Bootstrap}
\label{sec: Gaus and multiplier}

\subsection{Comparison bounds}

Let $X =(X_{1},\dots,X_{p})^{T}$ and $Y=(Y_{1},\dots,Y_{p})^{T}$ be centered Gaussian random vectors in $\RR^{p}$
with covariance matrices $\Sigma^{X} =(\sigma^{X}_{jk})_{1 \leq j,k \leq p}$ and
$\Sigma^{Y}=(\sigma^{Y}_{jk})_{1 \leq j,k \leq p}$, respectively.
The purpose of this section is to give error bounds on
the difference of the expectations of smooth functions and the distribution functions
of
\begin{equation*}
\max_{1 \leq j \leq p} X_{j} \quad \text{and} \quad \max_{1 \leq j \leq p} Y_{j}
\end{equation*}
in terms of $p$,
\begin{equation*}
\Delta :=\max_{1\leq j,k\leq p} | \sigma^{X}_{jk}-\sigma^{Y}_{jk} |, \  \text{and} \ a_p:=\Ep [ \max_{1 \leq j \leq p} (Y_{j}/\sigma^{Y}_{jj})].
\end{equation*}

The problem of comparing distributions of maxima is  of intrinsic difficulty since the maximum function $z=(z_{1},\dots,z_{p})^{T} \mapsto \max_{1 \leq j \leq p} z_{j}$ is non-differentiable.
To circumvent the problem, we use a smooth approximation of the maximum function. For $z=(z_{1},\dots,z_{p})^{T} \in \RR^{p}$, consider the function:
\begin{equation*}
F_{\beta}(z):=\beta^{-1}\log\left(\sum_{j=1}^{p}\exp(\beta z_{j})\right),
\end{equation*}
which approximates the maximum function, where $\beta > 0$ is the smoothing parameter that controls the level of approximation (we call this function the ``smooth max function'').
Indeed, an elementary calculation shows that for every $z \in \RR^{p}$,
\begin{equation}\label{eq: smooth max property}
0 \leq  F_{\beta}(z)- \max_{1 \leq j \leq p} z_{j} \leq \beta^{-1} \log p.
\end{equation}
This smooth max function arises in the definition of ``free energy" in spin glasses. See, for example, \cite{Talagrand2003} and \cite{Panchenko2013}.
Here is the first theorem of this section.

\begin{theorem}[Comparison bounds for smooth functions]
\label{theorem:comparison}
For every  $g \in C^2(\RR)$ with $\| g' \|_{\infty} \vee \| g'' \|_{\infty} < \infty$ and every $\beta>0$,
\begin{align*}
&| \Ep[g(F_{\beta}(X)) - g(F_{\beta}(Y))] |  \leq (\| g'' \|_{\infty}/2 +  \beta \| g' \|_{\infty})\Delta,\\
\intertext{and hence}
& | \Ep   [ g (\max_{1 \leq j \leq p}X_{j} ) - g(\max_{1 \leq j \leq p} Y_{j}) ] | \leq (\| g'' \|_{\infty}/2 +  \beta \| g' \|_{\infty}) \Delta + 2 \beta^{-1} \| g' \|_{\infty} \log p.
\end{align*}
\end{theorem}
\begin{proof}
See Section \ref{sec: proof for section 2}.
\end{proof}

\begin{remark}
Minimizing the second bound with respect to $\beta > 0$, we have
\begin{equation*}
| \Ep [ g (\max_{1 \leq j \leq p}X_{j} ) - g(\max_{1 \leq j \leq p} Y_{j})] |  \leq \| g'' \|_{\infty}\Delta/2+ 2\| g' \|_{\infty}\sqrt{2\Delta\log p}.
\end{equation*}
This result extends the work of \cite{Chatterjee2005b}, which derived the following Sudakov-Fernique type bound on the expectation of the difference between two Gaussian maxima:
\begin{equation*}
| \Ep [ \max_{1 \leq j \leq p}X_{j} - \max_{1 \leq j \leq p} Y_{j}] |  \leq  2\sqrt{2\Delta\log p}.
\end{equation*}
\end{remark}

Theorem \ref{theorem:comparison} is not applicable to functions of the form $g(z) = 1(z \leq x)$ and hence does not directly lead to a bound on the Kolmogorov distance between $\max_{1 \leq j \leq p} X_{j}$ and $\max_{1 \leq j \leq p} Y_{j}$ (recall that the Kolmogorov distance between (the distributions) of two real valued random variables
$\xi$ and $\eta$ is defined by $\sup_{x \in \RR} | \Pr ( \xi \leq x ) - \Pr (\eta \leq x) |$).
Nevertheless, we have the following bounds on the Kolmogorov distance. Recall $a_p=\Ep [ \max_{1 \leq j \leq p} (Y_{j}/\sigma^{Y}_{jj})]$.

\begin{theorem}[Comparison of distributions]\label{cor: distances Gaussian to Gaussian}
Suppose that $p \geq 2$ and $\sigma^{Y}_{jj} > 0$ for all $1 \leq j \leq p$.   Then
\begin{align}
&\sup_{x \in\RR} | \Pr ( \max_{1 \leq j \leq p} X_{j}\leq x ) -\Pr ( \max_{1 \leq j \leq p} Y_{j}\leq x ) | \notag \\
&\qquad \leq C \Delta^{1/3} \left  \{ (1 \vee a_{p}^{2} \vee \log (1/\Delta ) \right  \}^{1/3}  \log^{1/3} p,  \label{eq: K-distance}
\end{align}
where $C>0$ depends only on $\min_{1 \leq j \leq p} \sigma_{jj}^{Y}$ and $\max_{1 \leq j \leq p} \sigma_{jj}^{Y}$ (the right side is understood to be $0$ when $\Delta = 0$). Moreover, in the worst case, $a_{p} \leq \sqrt{2 \log p}$, so that
\[
\sup_{x \in\RR} | \Pr ( \max_{1 \leq j \leq p} X_{j}\leq x ) -\Pr ( \max_{1 \leq j \leq p} Y_{j}\leq x ) | \leq C' \Delta^{1/3} \{1 \vee \log (p/\Delta)\}^{2/3},
\]
where as before $C' > 0$ depends only on $\min_{1 \leq j \leq p} \sigma_{jj}^{Y}$ and $\max_{1 \leq j \leq p} \sigma_{jj}^{Y}$.
\end{theorem}
\begin{proof}
See Section \ref{sec: proof for section 2}.
\end{proof}

The first bound (\ref{eq: K-distance}) is generally sharper than the latter. To see this, consider the simple case where $a_{p} = O(1)$ as $p \to \infty$, which would happen, for example, when $Y_{1},\dots,Y_{p}$ come from discretization of a single continuous Gaussian process. Then the right side on (\ref{eq: K-distance}) is $o(1)$ if $\Delta (\log p) \log \log p= o(1)$, while the second bound requires $\Delta (\log p)^{2}= o(1)$.

\begin{remark}[On the proof strategy] Bounding the Kolmogorov distance between $\max_{1 \leq j \leq p}X_{j}$ and $\max_{1 \leq j \leq p}Y_{j}$ is  not immediate from Theorem \ref{theorem:comparison}  and this step relies on the anti-concentration inequality for the maximum of a  Gaussian random vector, which we will study in Section \ref{sec: anti-concentration}. 
More formally, by smoothing the indicator and maximum functions, we obtain from Theorem \ref{theorem:comparison} a bound of the following form:
\[
\inf_{\beta,\delta > 0} \{ \LL(\max_{1 \leq j \leq p} Y_{j},\beta^{-1}\log p+\delta) + C(\delta^{-2}+\beta\delta^{-1}) \Delta \},
\]
where $\LL(\max_{1 \leq j \leq p} Y_{j},\epsilon)$ is the L\'{e}vy concentration function for $\max_{1 \leq j \leq p}Y_{j}$ (see Definition \ref{def: levy concentration} in Section \ref{sec: anti-concentration} for the formal definition), and $\beta,\delta > 0$ are smoothing parameters (see equations (\ref{eq: Kdistance-intermediate}) and (\ref{eq: Kdistance-intermediate2}) in the proof of Theorem \ref{cor: distances Gaussian to Gaussian} given in Section \ref{sec: proof for section 2} for the derivation of the above bound). The bound (\ref{eq: K-distance}) then follows from bounding the L\'{e}vy concentration function by using the anti-concentration inequality derived in Section \ref{sec: anti-concentration}, and optimizing the bound with respect to $\beta,\delta$. 

The proof of Theorem \ref{cor: distances Gaussian to Gaussian} is substantially different from the (``textbook'') proof of classical Slepian's inequality. The simplest form of Slepian's inequality states that
\begin{equation*}
\Pr ( \max_{1 \leq j \leq p} X_{j} \leq x) \leq \Pr ( \max_{1 \leq j \leq p} Y_{j} \leq x), \ \forall x \in \RR,
\end{equation*}
whenever $\sigma_{jj}^{X} = \sigma_{jj}^{Y}$ and $\sigma_{jk}^{X} \leq \sigma_{jk}^{Y}$ for all $1 \leq j,k \leq p$.
This inequality is immediately deduced from the following expression:
\begin{multline}
\Pr ( \max_{1 \leq j \leq p} X_{j}\leq x ) -\Pr ( \max_{1 \leq j \leq p} Y_{j}\leq x ) \\
= \sum_{1 \leq j < k \leq p} (\sigma_{jk}^{X} - \sigma_{jk}^{Y}) \int_{0}^{1} \left \{ \int_{-\infty}^{x} \cdots \int_{-\infty}^{x} \frac{\partial^{2} f_{t}(z)}{\partial z_{j}\partial z_{k}} dz \right \} dt, \label{eq: slepian expression}
\end{multline}
where $\sigma_{jj}^{X}  = \sigma_{jj}^{Y},1 \leq \forall j \leq p$,  is assumed. Here $f_{t}$ denotes the density function of $N(0,t \Sigma^{X} + (1-t) \Sigma^{Y})$. See \cite{Leadbetter1983}, p.82, for this expression.
The expression (\ref{eq: slepian expression}) is of  importance and indeed a source of many interesting probabilistic results (see, for example, \cite{LS02} and \cite{Y09} for recent related works).
It is not clear (or at least non-trivial), however,  whether a bound similar in nature to Theorem \ref{cor: distances Gaussian to Gaussian} can be deduced from the expression (\ref{eq: slepian expression}) when there is no restriction on the covariance matrices except for the condition that $\sigma_{jj}^{X}  = \sigma_{jj}^{Y},1 \leq \forall j \leq p$, and here we take the different route.
\end{remark}

The key features of Theorem \ref{cor: distances Gaussian to Gaussian} are: (i) the bound on the Kolmogorov distance between the maxima of Gaussian random vectors in $\RR^{p}$ depends on the dimension $p$ only through  $\log p$ and the maximum difference of the covariance matrices $\Delta$, and (ii) it allows for arbitrary covariance matrices for $X$ and $Y$ (except for the nondegeneracy condition that $\sigma_{jj}^{Y} > 0, \ 1 \leq \forall j \leq p$). These features have an important implication to statistical applications, as discussed below.

\subsection{Conditional multiplier central limit theorem}

Consider the following problem. Suppose that $n$ independent centered random vectors in $\RR^{p}$ of observations $Z_{1},\dots,Z_{n}$ are given. Here
$Z_{1},\dots,Z_{n}$ are generally non-Gaussian, and the dimension $p$ is allowed to increase with $n$ (that is, the case where $p=p_{n} \to \infty$ as $n \to \infty$ is allowed).
We suppress the possible dependence of $p$ on $n$ for the notational convenience.   Suppose that each $Z_{i}$ has a finite covariance matrix $\Ep [ Z_{i} Z_{i}^{T} ]$.
Consider the following normalized sum:
\begin{equation*}
S_{n} := (S_{n,1},\dots,S_{n,p})^{T} = \frac{1}{\sqrt{n}} \sum_{i=1}^{n} Z_{i}.
\end{equation*}
The problem here is to approximate the distribution of $\max_{1 \leq j \leq p} S_{n,j}$.

Statistics of this form arise frequently in modern statistical applications. The exact distribution of $\max_{1 \leq j \leq p} S_{n,j}$ is generally unknown.
An intuitive idea to approximate the distribution of $\max_{1 \leq j \leq p} S_{n,j}$ is to use the Gaussian approximation. Let $V_{1},\dots,V_{n}$ be independent Gaussian random vectors in $\RR^{p}$ such that $V_{i} \sim N(0,\Ep [ Z_{i} Z_{i}^{T} ])$, and define
\begin{equation*}
T_{n} := (T_{n,1},\dots,T_{n,p}) := \frac{1}{\sqrt{n}} \sum_{i=1}^{n} V_{i} \sim N(0,n^{-1} {\textstyle \sum}_{i=1}^{n} \Ep [Z_{i}Z_{i}^{T}] ).
\end{equation*}
It is expected that the distribution of $\max_{1 \leq j \leq p} T_{n,j}$ is close to that of $\max_{1 \leq j \leq p} S_{n,j}$ in the following sense:
\begin{equation}
 \sup_{x \in \RR} | \Pr ( \max_{1 \leq j \leq p} S_{n,j} \leq x ) - \Pr ( \max_{1 \leq j \leq p} T_{n,j} \leq x )   | \to 0, \ n \to \infty. \label{eq: gaussian-approx}
\end{equation}
When $p$ is fixed, (\ref{eq: gaussian-approx}) will follow from the classical Lindeberg-Feller central limit theorem, subject to the Lindeberg conditions. The recent paper by \cite{CCK12} established conditions under which this Gaussian approximation (\ref{eq: gaussian-approx}) holds even when $p$ is comparable or much larger than $n$.
For example, \cite{CCK12} proved that if $c_{1} \leq n^{-1} \sum_{i=1}^{n} \Ep [ Z_{ij}^{2} ] \leq C_{1}$ and $\Ep [ \exp (| Z_{ij} |/ C_{1} ) ] \leq 2$ for all $1 \leq i \leq n$ and $1 \leq j \leq p$ for some $0 < c_{1} < C_{1}$,
then (\ref{eq: gaussian-approx}) holds as long as $\log p = o(n^{1/7})$.

The Gaussian approximation (\ref{eq: gaussian-approx}) is in itself an important step, but in the general case where the covariance matrix $n^{-1} \sum_{i=1}^{n} \Ep [ Z_{i} Z_{i}^{T} ]$ is unknown, it is not directly applicable for purposes of statistical inference. In such cases, the following {\em multiplier bootstrap} procedure will be useful. Let $\eta_{1},\dots,\eta_{n}$ be independent standard Gaussian random variables independent of $Z_{1}^{n} := \{ Z_{1},\dots,Z_{n} \}$. Consider the following randomized sum:
\begin{equation*}
S_{n}^{\eta} := (S_{n,1}^{\eta}, \dots, S_{n,p}^{\eta} )^{T} := \frac{1}{\sqrt{n}} \sum_{i=1}^{n} \eta_{i} Z_{i}.
\end{equation*}
Since conditional on $Z_{1}^{n}$,
\begin{equation*}
S_{n}^{\eta} \sim N(0,n^{-1} {\textstyle \sum}_{i=1}^{n} Z_{i}Z_{i}^{T}),
\end{equation*}
it is natural to expect that the conditional distribution of $\max_{1 \leq j \leq p} S_{n,j}^{\eta}$ is ``close'' to the distribution of $\max_{1 \leq j \leq p} T_{n,j}$ and hence that of $\max_{1 \leq j \leq p} S_{n,j}$. Note here that the conditional distribution of $S^{\eta}_{n}$ is completely known, which makes this distribution useful for purposes of statistical inference. The following proposition makes this intuition rigorous.

\begin{proposition}[Conditional multiplier central limit theorem]
\label{prop: multiplier bootstrap}
Work with the setup as described above. Suppose that $p \geq 2$ and  there are some constants $0 < c_{1} < C_{1}$ such that
$c_{1} \leq n^{-1} \sum_{i=1}^{n} \Ep [ Z_{ij}^{2} ] \leq C_{1}$ for all $1 \leq j \leq p$. Moreover, suppose that  $\hat{\Delta} := \max_{1 \leq j,k \leq p} | n^{-1} \sum_{i=1}^{n} (Z_{ij}Z_{ik} - \Ep [ Z_{ij}Z_{ik} ])|$ obeys the following conditions: as $n \to \infty$,
\begin{equation}
\hat{\Delta} ( \Ep [\max_{1 \leq j \leq p} T_{n,j}] )^{2} \log p = o_{\Pr}(1), \ \hat{\Delta} (\log p) (1 \vee \log \log p)= o_{\Pr}(1). \label{eq: growth}
\end{equation}
Then we have
\begin{equation}
\sup_{x \in \RR} | \Pr ( \max_{1 \leq j \leq p} S_{n,j}^{\eta} \leq x \mid Z_{1}^{n} ) - \Pr ( \max_{1 \leq j \leq p} T_{n,j} \leq x ) | \stackrel{\Pr}{\to} 0, \ \text{as} \ n \to \infty. \label{eq: cmclt}
\end{equation}
Here recall that $p$ is allowed to increase with $n$.
\end{proposition}

\begin{proof}
Follows immediately from Theorem \ref{cor: distances Gaussian to Gaussian}.
\end{proof}

We call this result a ``conditional multiplier central limit theorem,'' where the terminology follows that in  empirical process theory. See \cite{VW96}, Chapter 2.9.
The notable features of this proposition, which inherit from the features of Theorem \ref{cor: distances Gaussian to Gaussian} discussed above, are:  (i) (\ref{eq: cmclt}) can hold even when $p$ is much larger than $n$, and (ii) it allows for arbitrary covariance matrices for $Z_{i}$ (except for the mild scaling condition that $c_{1} \leq n^{-1} \sum_{i=1}^{n} \Ep [ Z_{ij}^{2} ] \leq C_{1}$). The second point is clearly desirable in statistical applications as the information on the true covariance structure is generally (but not always) unavailable.
For the first point, we have the following estimate on $\Ep [ \hat{\Delta} ]$.

\begin{lemma}
\label{lem: Delta estimate}
Let $p \geq 2$. There exists a universal constant $C > 0$ such that
\begin{equation*}
\Ep [ \hat{\Delta} ] \leq C \left [ \max_{1 \leq j \leq p} (n^{-1} {\textstyle \sum}_{i=1}^{n} \Ep [ Z_{ij}^{4} ])^{1/2} \sqrt{\frac{ \log p}{n}}  + ( \Ep [ \max_{1 \leq i \leq n} \max_{1 \leq j \leq p} Z_{ij}^{4} ] )^{1/2}\frac{\log p}{n} \right] .
\end{equation*}
\end{lemma}
\begin{proof}
See the Appendix.
\end{proof}
Hence with help of Lemma 2.2.2 in \cite{VW96}, we can find various primitive conditions under which (\ref{eq: growth}) holds.

\begin{example} Consider the following examples.  Here for the sake of simplicity, we use the worst case bound $\Ep [ \max_{1 \leq j \leq p} T_{n,j} ] \leq \sqrt{2 C_{1} \log p}$, so that conditions (\ref{eq: growth}) reduce to $\hat{\Delta} = o_{\Pr}((\log p)^{-2})$.

Case (a): Suppose that $\Ep [ \exp (| Z_{ij} |/ C_{1} ) ] \leq 2$ for all $1 \leq i \leq n$ and $1 \leq j \leq p$ for some $C_{1} > 0$.
In this case, it is not difficult to verify that $\hat{\Delta} = o_{\Pr}((\log p)^{-2})$ as soon as $\log p = o(n^{1/5})$.

Case (b): Another type of $Z_{ij}$ which arises in regression applications is of the form $Z_{ij} = \eps_{i} x_{ij}$ where $\eps_{i}$ are stochastic with $\Ep [ \epsilon_{i} ] = 0$ and $\max_{1 \leq i \leq n}\Ep[ | \eps_{i} |^{4q} ] = O(1)$ for some $q \geq 1$, and $x_{ij}$ are non-stochastic (typically, $\eps_{i}$ are ``errors'' and $x_{ij}$ are ``regressors'').  Suppose that $x_{ij}$ are normalized in such a way that $n^{-1} \sum_{i=1}^{n} x_{ij}^{2} = 1$, and there are bounds $B_{n} \geq 1$ such that $\max_{1 \leq i \leq n} \max_{1 \leq j \leq p} | x_{ij} | \leq B_{n}$, where we allow $B_{n} \to \infty$. In this case, $\hat{\Delta} = o_{\Pr}((\log p)^{-2})$ as soon as
\begin{equation*}
\max \{ B_{n}^{2} (\log p)^{5},  B_{n}^{4q/(2q-1)} (\log p)^{6q/(2q-1)}  \} = o(n),
\end{equation*}
since $\max_{1 \leq j \leq p} ( n^{-1}\sum_{i=1}^{n} \Ep [ (\eps_{i}  x_{ij})^{4} ]) \leq B_{n}^{2} \max_{1 \leq i \leq n} \Ep [ \eps_{i}^{4} ] = O(B_{n}^{2})$ and $\Ep [ \max_{1 \leq i \leq n} \max_{1 \leq j \leq p} (\eps_{i} x_{ij})^{4} ] \leq B_{n}^{4} \Ep[\max_{1 \leq i \leq n} \eps_{i}^{4} ] = O(n^{1/q} B_{n}^{4})$.

Importantly, in these examples, for (\ref{eq: cmclt}) to hold, $p$ can increase  exponentially in some fractional power of $n$.
\end{example}

\section{Anti-concentration Bounds}
\label{sec: anti-concentration}

The following theorem provides bounds on the {\em L\'{e}vy concentration function}
of the maximum of a Gaussian random vector in $\RR^{p}$, where the terminology is borrowed from \cite{RV09}.

\begin{definition}[\cite{RV09}, Definition 3.1]
\label{def: levy concentration}
The {\em L\'{e}vy concentration function} of a real valued random variable $\xi$ is defined for $\epsilon > 0$ as
\begin{equation*}
\LL (\xi,\epsilon) = \sup_{x \in \RR} \Pr ( | \xi - x | \leq \epsilon ).
\end{equation*}
\end{definition}


 \begin{theorem}[Anti-concentration]\label{thm: anticoncentration}
Let $(X_{1},\dots,X_{p})^{T}$ be a  centered Gaussian random vector in $\RR^{p}$ with $\sigma_{j}^{2} := \Ep [ X_{j}^{2} ]>0$ for all $1 \leq j \leq p$. Moreover, let
$\underline{\sigma} := \min_{1 \leq j \leq p} \sigma_{j}, \overline{\sigma} := \max_{1 \leq j \leq p} \sigma_{j}$, and $a_{p} := \Ep [ \max_{1 \leq j \leq p} (X_{j}/\sigma_{j}) ]$.

(i) If the variances are all equal, namely $\underline{\sigma} = \overline{\sigma} = \sigma$, then for every $\epsilon > 0$,
\begin{equation*}
\LL (\max_{1 \leq j \leq p} X_{j},\epsilon) \leq 4 \epsilon (a_p + 1)/\sigma.
\end{equation*}

(ii)  If the variances are not equal, namely $\underline{\sigma} < \overline{\sigma}$, then for every $\epsilon > 0$,
\begin{equation*}
\LL (\max_{1 \leq j \leq p} X_{j},\epsilon)  \leq C \epsilon \{ a_p + \sqrt{1 \vee \log (\underline{\sigma}/\epsilon)} \},
\end{equation*}
where $C>0$ depends only on $\underline{\sigma}$ and $\overline{\sigma}$.
\end{theorem}

Since $X_{j}/\sigma_{j} \sim N(0,1)$, by a standard calculation, we have $a_{p} \leq \sqrt{2 \log p}$ in the worst case
(see, for example, \cite{Talagrand2003}, Proposition 1.1.3), so that
the following simpler corollary  follows immediately from Theorem \ref{thm: anticoncentration}.

 \begin{corollary}\label{cor: anticoncentration}
Let $(X_{1},\dots,X_{p})^{T}$ be a  centered Gaussian random vector in $\RR^{p}$ with $\sigma_{j}^{2} := \Ep [ X_{j}^{2} ]  > 0$ for all $1 \leq j \leq p$.
Let $\underline{\sigma} := \min_{1 \leq j \leq p} \sigma_{j}$ and $\overline{\sigma} := \max_{1 \leq j \leq p} \sigma_{j}$. Then for every $\epsilon > 0$,
\begin{equation*}
\LL (\max_{1 \leq j \leq p} X_{j},\epsilon)  \leq C \epsilon  \sqrt{1 \vee \log (p/\epsilon )},
\end{equation*}
where $C>0$ depends only on $\underline{\sigma}$ and $\overline{\sigma}$.
When $\sigma_{j}$ are all equal,  $\log (p/\epsilon)$ on the right side can be replaced by $\log p$.
\end{corollary}


\begin{remark}[Anti-concentration vs. small ball probabilities]
The problem of bounding the L\'{e}vy concentration function $\LL (\max_{1 \leq j \leq p} X_{j},\epsilon)$ is qualitatively different from
the problem of bounding $\Pr ( \max_{1 \leq j \leq p} | X_{j} | \leq x )$. For a survey on the latter problem, called the ``small ball problem'', we refer the reader to \cite{LS01}.
\end{remark}

\begin{remark}[Concentration vs. anti-concentration]
\label{rem: concentration vs anticoncentration}
{\em Concentration inequalities} refer to inequalities bounding $\Pr (| \xi - x | > \epsilon )$ for a random variable $\xi$ (typically $x$ is the mean or median of $\xi$). See the monograph \cite{L01} for a study of the concentration of measure phenomenon.
{\em Anti-concentration inequalities} in turn refer to reverse inequalities, that is, inequalities bounding $\Pr (| \xi - x | \leq \epsilon )$. Theorem \ref{thm: anticoncentration} provides anti-concentration inequalities for $\max_{1 \leq j \leq p} X_{j}$.
\cite{VR07} remarked that ``concentration is better understood than anti-concentration''. In the present case, the Gaussian concentration inequality (see \cite{L01}, Theorem 7.1) states that
\begin{equation*}
\Pr ( | \max_{1 \leq j \leq p} X_{j} - \Ep [ \max_{1 \leq j \leq p} X_{j} ] | \geq r) \leq 2 e^{-r^{2}/(2 \overline{\sigma}^{2})}, \ r > 0,
\end{equation*}
where the mean can be replace by the median. This inequality is well known and dates back to \cite{B75} and \cite{ST78}.
To the best of our knowledge, however, the reverse inequalities in Theorem \ref{thm: anticoncentration} were not known and are new.
\end{remark}

\begin{remark}[Anti-concentration for maximum  of moduli, $\max_{1 \leq j \leq p} | X_{j} |$]
Versions of Theorem \ref{thm: anticoncentration} and Corollary \ref{cor: anticoncentration} continue to hold for $\max_{1 \leq j \leq p} | X_{j} |$.
That is, for example, when $\sigma_{j}$ are all equal ($\sigma_{j} = \sigma$), $\LL(\max_{1 \leq j \leq p} | X_{j} |, \epsilon ) \leq 4 ( a_{p}'+1)/\sigma$, where $a_{p}' := \Ep [ \max_{1 \leq j \leq p} | X_{j} |/\sigma ]$.  To see this, observe that $\max_{1 \leq j \leq p} | X_{j} | = \max_{1 \leq j \leq 2p} X_{j}'$ where $X_{j}'=X_{j}$ for $j=1,\dots,p$ and $X_{p+j}' = - X_{j}$ for $j=1,\dots,p$. Hence we may apply Theorem \ref{thm: anticoncentration} to $(X_{1}',\dots,X_{2p}')^{T}$ to obtain the desired conclusion.
\end{remark}

\begin{remark}[A sketch of the proof of Theorem \ref{thm: anticoncentration}]
For the reader's convenience, we provide a sketch of the proof of Theorem \ref{thm: anticoncentration}. We focus here on the simple case where all the variances are equal to one $(\sigma_{1}=\cdots=\sigma_{p}=1$). Then the distribution of $Z=\max_{1 \leq j \leq p} X_{j}$ is absolutely continuous and its density can be written as $\phi (z) G(z)$ where the map $z \mapsto G(z)$ is non-decreasing. Consequently, it is then not difficult to see that $G(z) \leq \Pr(Z > z)/\{ 1- \Phi(z) \} \leq 2 (z \vee 1) e^{-(z-a_{p})_{+}^{2}/2}/\phi(z)$, where the second inequality follows from Mill's inequality combined with  the Gaussian concentration inequality. Hence the density of $Z$ is bounded by $2 (z \vee 1) e^{-(z-a_{p})_{+}^{2}/2}$, which immediately leads to the bound $\LL (\max_{1 \leq j \leq p} X_{j},\epsilon) \leq 4 (a_{p}+1)\epsilon$.
\end{remark}

In a trivial example where $p=1$, it is immediate to see that $\Pr ( | X_{1} - x|  \leq   \epsilon  ) \leq \epsilon \sqrt{2/(\pi \sigma_{1}^{2})}$.
A non-trivial case is the situation where $p \to \infty$. In such a case, it is typically not known whether $\max_{1 \leq j \leq p} X_{j}$ has a limiting distribution as $p \to \infty$, even after normalization (recall that except for $\underline{\sigma} > 0$, we allow for general covariance structures between $X_{1},\dots,X_{p}$), and therefore it is not trivial at all whether,
for every sequence $\epsilon = \epsilon_{p} \to 0$ (or at some rate), $\LL (\max_{1 \leq j \leq p} X_{j},\epsilon) \to 0$ or how fast $\epsilon = \epsilon_{p} \to 0$ should be to guarantee that $\LL (\max_{1 \leq j \leq p} X_{j},\epsilon) \to 0$.
Theorem \ref{thm: anticoncentration} answers this question with explicit, non-asymptotic bounds.

 Importantly, the bounds in Theorem \ref{thm: anticoncentration} are {\em dimension-free} in the sense that, similar to the Gaussian concentration inequality, they depend on the dimension $p$ only through $a_{p}$ -- the expectation of the maximum of the (normalized) Gaussian random vector. Hence Theorem \ref{thm: anticoncentration} admits direct extensions to the infinite dimensional case, namely separable Gaussian processes, as long as the corresponding expectation is finite. See our companion paper \cite{CCK13} for formal treatments and applications of this extension.

The presence of $a_{p}$ on the bounds is essential and can not be removed in general, as the following example suggests.
This shows that there does not exist a substantially sharper estimate of the universal bound of the concentration function than that given in Theorem \ref{thm: anticoncentration}.
Potentially, there could be refinements but they would have to rely on the particular (hence
non-universal) features of the covariance structure between $X_{1},\dots,X_{p}$.

\begin{example}[Partial converse of Theorem \ref{thm: anticoncentration}]
\label{ex1}

Let $X_{1},\dots,X_{p}$ be independent standard Gaussian random variables.
By Theorem 1.5.3 of \cite{Leadbetter1983},  as $p \to \infty$,
\begin{equation}
b_{p} (\max_{1 \leq j \leq p}X_{j} - d_{p}) \stackrel{d}{\to} G(0,1), \label{weak}
\end{equation}
where
\begin{equation*}
b_{p} := \sqrt{2 \log p},   \ \ d_{p} := b_{p} - \frac{ \log(4 \pi) + \log \log p }{2 b_{p}},
\end{equation*}
and $G(0,1)$ denotes the standard Gumbel distribution, that is, the distribution having the density $g(x) = e^{-x} e^{-e^{-x}}$ for $x \in \mathbb{R}$.
In fact, we can show that the density of $b_{p} (\max_{1 \leq j \leq p}X_{j} - d_{p})$ converges to that of $G(0,1)$ locally uniformly.
To see this, we begin with noting that the density of $b_{p} (\max_{1 \leq j \leq p}X_{j} - d_{p})$ is given by
\begin{equation*}
g_{p}(x) = \frac{p}{b_{p}}\phi(d_{p} + b_{p}^{-1}x) [\Phi(d_{p}+b_{p}^{-1}x)]^{p-1}.
\end{equation*}
Pick any $x \in \mathbb{R}$. Since, by the weak convergence result (\ref{weak}),
\begin{equation*}
[\Phi(d_{p}+b_{p}^{-1}x)]^p =\Pr(b_{p} (\max_{1 \leq j \leq p}X_{j} - d_{p})\leq x) \to e^{-e^{-x}}, \ p \to \infty,
\end{equation*}
we have $[\Phi(d_{p}+b_{p}^{-1}x)]^{p-1} \to e^{-e^{-x}}$. Hence it remains to show that
\begin{equation*}
\frac{p}{b_{p}}\phi(d_{p} + b_{p}^{-1}x) \to e^{-x}.
\end{equation*}
Taking the logarithm of the left side yields
\begin{equation}
\log p -\log b_{p} - \log (\sqrt{2\pi}) - (d_{p}+b_{p}^{-1}x)^{2}/2. \label{log}
\end{equation}
Expanding $(d_{p}+b_{p}^{-1}x)^{2}$ gives that
\begin{equation*}
d_{p}^{2} + 2 d_{p} b_{p}^{-1}x + b_{p}^{-2}x^{2} = b_{p}^{2} - \log \log p - \log (4\pi) + 2x + o(1), \ p \to \infty,
\end{equation*}
by which we have $(\ref{log}) = -x + o(1)$. This shows that $g_{p}(x) \to g(x)$ for all $x \in \RR$.
Moreover, this convergence takes place locally uniformly in $x$, that is, for every $K > 0$, $g_{p}(x) \to g(x)$ uniformly in $x \in [-K, K]$.

On the other hand, the density of $\max_{1 \leq j \leq p} X_{j}$ is given by $f_{p}(x) = p \phi(x) [ \Phi(x) ]^{p-1}$.
By this form, for every $K > 0$, there exist a constant $c > 0$ and a positive integer $p_{0}$ depending only on $K$ such that for $p \geq p_{0}$,
\begin{equation*}
\inf_{x \in [ d_{p} - K b_{p}^{-1}, d_{p} +K b_{p}^{-1}]} b_{p}^{-1} f_{p}(x) = \inf_{x \in [-K,K]} g_{p}(x) \geq \inf_{x \in [-K,K]} g(x) + o(1) \geq c,
\end{equation*}
which shows that for $p \geq p_{0}$,
\begin{equation*}
f_{p}(x) \geq c b_{p}, \ \forall x \in [ d_{p} - Kb_{p}^{-1}, d_{p} +K b_{p}^{-1}].
\end{equation*}
Therefore, we conclude that for $p \geq p_{0}$,
\begin{equation*}
\Pr(|\max_{1 \leq j \leq p} X_{j} - d_{p} | \leq \epsilon) = \int_{d_{p}-\epsilon}^{d_{p}+\epsilon} f_{p}(x) dx \geq 2 c \epsilon b_{p}, \ \forall \epsilon \in [ 0, K b_{p}^{-1} ].
\end{equation*}

By the Gaussian maximal inequality and Lemma 2.3.15 of \cite{D99}, we have
\begin{equation*}
\sqrt{\log p}/12  \leq \mathbb{E}[ \max_{1 \leq j \leq p} X_{j}] \leq \sqrt{2 \log p}.
\end{equation*}
Hence,  by the previous result, for every  $K' > 0$, there exist a constant $c' > 0$ and and a positive integer $p_{0}'$ depending only on $K'$ such that for $p \geq p_{0}'$,
$a_{p} \geq 1$ and
\begin{equation*}
\LL (\max_{1 \leq j \leq p} X_{j},\epsilon) \geq \Pr(|\max_{1 \leq j \leq p} X_{j} - d_{p} | \leq \epsilon) \geq c' \epsilon a_{p}, \ \forall \epsilon \in [0, K'a^{-1}_{p} ].
\end{equation*}
\qed
\end{example}

\section{Proofs for Section \ref{sec: Gaus and multiplier}}
\label{sec: proof for section 2}

\subsection{Proof of Theorem  \ref{theorem:comparison}}

Here for a smooth function $f: \RR^{p} \to \RR$, we write $\partial_{j} f (z) = \partial f(z)/\partial z_{j}$ for $z = (z_{1},\dots,z_{p})^{T}$.
We shall use the following version of Stein's identity.

\begin{lemma}[Stein's identity]
\label{lemma: talagrand}
Let $W =(W_{1},\dots,W_{p})^{T}$ be a centered Gaussian random vector in $\RR^{p}$.
Let $f: \RR^{p} \to \RR$ be a $C^1$-function such that $\Ep [ |\partial_{j} f(W) | ] < \infty$ for all $1 \leq j \leq p$. Then for every $1 \leq j \leq p$,
\begin{equation*}
\Ep[W_{j}f(W)]=\sum_{k=1}^p\Ep[W_{j} W_{k}] \Ep [\partial_k f (W)].
\end{equation*}
\end{lemma}
\begin{proof}[Proof of Lemma \ref{lemma: talagrand}]
See Section A.6 of \cite{Talagrand2003}; also \cite{ChenGoldsteinShao2011} and \cite{Stein1981}.
\end{proof}

We will use the following properties of the smooth max function.

\begin{lemma}
\label{lemma: smooth max property}
For every $1 \leq j,k \leq p$,
\begin{equation*}
 \partial_{j} F_{\beta}(z) = \pi_{j}(z), \quad \partial_{j} \partial_k F_{\beta}(z) = \beta w_{jk} (z),
\end{equation*}
where
\begin{equation*}
\pi_{j}(z)  :=    e^{\beta z_{j}}/{\textstyle \sum}_{m=1}^pe^{\beta z_m}, \ w_{jk}(z) := 1 (j = k) \pi_{j} (z)- \pi_{j}(z) \pi_{k} (z).
\end{equation*}
Moreover,
\begin{equation*}
\pi_{j}(z) \geq 0,  \  {\textstyle \sum}_{j=1}^p \pi_{j} (z) = 1, \ {\textstyle \sum}_{j,k=1}^p |w_{jk}(z)|  \leq 2.
\end{equation*}
\end{lemma}

\begin{proof}[Proof of Lemma \ref{lemma: smooth max property}]
The first property was noted in \cite{Chatterjee2005b}.  The other properties follow from a direct calculation.
\end{proof}

\begin{lemma}\label{lemma: second deriv m}
Let $m := g \circ F_{\beta}$ with $g \in C^{2}(\RR)$. Then for every $1 \leq j,k \leq p$,
\begin{align*}
\partial_{j} \partial_k m(z) =   (g'' \circ F_\beta) (z) \pi_{j} (z) \pi_{k} (z)+ \beta (g' \circ F_\beta) (z) w_{jk}(z),
\end{align*}
where $\pi_{j}$ and $w_{jk}$ are defined in Lemma \ref{lemma: smooth max property}.

\end{lemma}
\begin{proof}[Proof of lemma \ref{lemma: second deriv m}]
The proof follows from a direct calculation.
\end{proof}

\begin{proof}[Proof of Theorem \ref{theorem:comparison}]
Without loss of generality, we may assume that $X$ and $Y$ are independent, so that $\Ep[X_{j} Y_k]=0$ for all $1 \leq j,k \leq p$.
Consider the following Slepian interpolation between $X$ and $Y$:
\begin{equation*}
Z(t):=\sqrt{t} X+\sqrt{1-t} Y, \ t \in [0,1].
\end{equation*}
Let $m := g \circ F_\beta$ and $\Psi(t):=\Ep[m (Z(t))]$.
Then
\begin{equation*}
|\Ep[m(X)]-\Ep[ m(Y) ]|= | \Psi (1) - \Psi (0)| =\left |\int_{0}^{1} \Psi'(t) dt\right|.
\end{equation*}
Here we have
\begin{align*}
\Psi'(t) & = \frac{1}{2} \sum_{j=1}^p \Ep [ \partial_{j} m(Z(t)) (t^{-1/2}X_{j}-(1-t)^{-1/2}Y_{j}) ]  \\
& = \frac{1}{2} \sum_{j=1}^p \sum_{k=1}^p (\sigma_{jk}^{X}-\sigma_{jk}^{Y}) \Ep[\partial_{j} \partial_k m (Z(t)) ],
\end{align*}
where the second equality follows from applying Lemma \ref{lemma: talagrand} to $W = (t^{-1/2}X_{j}-(1-t)^{-1/2}Y_{j}, Z(t)^{T})^{T}$
and $f(W)=\partial_{j} m(Z(t))$. Hence
\begin{align*}
\left |\int_{0}^{1} \Psi'(t) dt\right|
&\leq \frac{1}{2} \sum_{j,k=1}^{p} | \sigma_{jk}^{X}-\sigma_{jk}^{Y} | \left | \int_0^1 \Ep [\partial_{j} \partial_k m(Z(t))] dt \right| \\
&\leq \frac{1}{2} \max_{1 \leq j,k \leq p} | \sigma_{jk}^{X}-\sigma_{jk}^{Y} | \int_0^1 \sum_{j,k=1}^{p} \left | \Ep[\partial_{j} \partial_k m(Z(t))]\right | dt   \\
&= \frac{\Delta}{2} \int_0^1 \sum_{j,k=1}^{p} \left | \Ep[\partial_{j} \partial_k m(Z(t))]\right | dt.
\end{align*}
By Lemmas \ref{lemma: smooth max property} and \ref{lemma: second deriv m},
\begin{equation*}
\sum_{j,k=1}^{p} |\partial_{j} \partial_k m(Z(t))|   \leq   | (g'' \circ F_{\beta})(Z(t))| + 2 \beta | (g' \circ F_{\beta})(Z(t))|.
\end{equation*}
Therefore, we have
\begin{align*}
&\left|\Ep[g(F_{\beta}(X))-g(F_{\beta}(Y))]\right| \\
&\leq  \Delta \times  \left \{  \frac{1}{2} \int_0^1 \Ep[ | (g'' \circ F_{\beta})(Z(t))| ] dt +  \beta \int_0^1 \Ep [| (g' \circ F_{\beta})(Z(t))| ]dt\right \}  \\
&\leq  \Delta ( \| g'' \|_{\infty}/2 + \beta \| g' \|_{\infty}),
\end{align*}
which leads to the first assertion. The second assertion follows from the inequality (\ref{eq: smooth max property}). This completes the proof.
\end{proof}

\subsection{Proof of Theorem \ref{cor: distances Gaussian to Gaussian}}
The final assertion follows from the inequality $a_{p} \leq \sqrt{2 \log p}$ (see, for example, \cite{Talagrand2003}, Proposition 1.1.3).
Hence we prove (\ref{eq: K-distance}).
We first note that we may assume that $0 < \Delta < 1$ since otherwise the proof is trivial (take $C \geq 2$ in (\ref{eq: K-distance})).
In what follows, let  $C>0$ be a generic constant that depends only on $\min_{1 \leq j \leq p} \sigma_{jj}^{Y}$ and $\max_{1 \leq j \leq p} \sigma_{jj}^{Y}$, and its value may change from place to place.
For $\beta > 0$, define $e_{p,\beta} := \beta^{-1} \log p$.
Consider and fix a $C^{2}$-function $g_0:\RR \to [0,1]$ such that $g_0(t)=1$ for $t \leq 0$ and $g_0(t)=0$ for $t \geq 1$. For example, we may take
\begin{equation*}
g_{0}(t) =
\begin{cases}
0, & t \geq 1, \\
30\int_{t}^{1} s^{2}(1-s)^{2} ds, & 0 < t < 1, \\
1, & t \leq 0.
\end{cases}
\end{equation*}
For given $x \in \RR, \beta > 0$ and $\delta > 0$,  define $g_{x,\beta,\delta} (t)=g_0(\delta^{-1}(t-x-e_{p,\beta}))$.  For this function $g_{x,\beta,\delta}$, $\| g_{x,\beta,\delta}' \|_{\infty} = \delta^{-1} \| g_{0}' \|_{\infty}$ and $\| g_{x,\beta,\delta}'' \|_{\infty} = \delta^{-2} \| g_{0}'' \|_{\infty}$. Moreover,
\begin{equation}
1(t \leq x+e_{p,\beta}) \leq g_{x,\beta,\delta}(t) \leq 1(t \leq x + e_{p,\beta} + \delta), \ \forall t \in \RR. \label{eq: property of g}
\end{equation}

For arbitrary $x \in \RR, \beta > 0$ and $\delta > 0$, observe that
\begin{align}
\Pr ( \max_{1 \leq j \leq p} X_{j}   \leq x )
&\leq \Pr ( F_{\beta} (X) \leq x + e_{p,\beta}) \leq \Ep[g_{x,\beta,\delta}(F_\beta(X))] \notag \\
&\leq \Ep[g_{x,\beta,\delta}(F_\beta(Y))] + C(\delta^{-2}+\beta \delta^{-1})\Delta \notag  \\
&\leq  \Pr ( F_\beta (Y) \leq x+e_{p,\beta}+\delta ) + C(\delta^{-2}+\beta \delta^{-1})\Delta \notag \\
&\leq \Pr (\max_{1 \leq j \leq p} Y_{j} \leq x+e_{p,\beta}+\delta ) + C(\delta^{-2}+\beta \delta^{-1})\Delta, \label{eq: Kdistance-intermediate}
\end{align}
where the first inequality follows from the inequality (\ref{eq: smooth max property}), the second from the inequality (\ref{eq: property of g}),  the third from Theorem \ref{theorem:comparison}, the fourth from the inequality (\ref{eq: property of g}), and the last from the inequality (\ref{eq: smooth max property}).
We wish to compare $\Pr (\max_{1 \leq j \leq p} Y_{j} \leq x+e_{p,\beta}+\delta) $ with $\Pr (\max_{1 \leq j \leq p} Y_{j} \leq x ) $, and this is where the anti-concentration inequality plays its role.
By Theorem \ref{thm: anticoncentration}, we have
\begin{align}
&\Pr (\max_{1 \leq j \leq p} Y_{j} \leq x+e_{p,\beta}+\delta ) -  \Pr (\max_{1 \leq j \leq p} Y_{j} \leq x ) \notag \\
&= \Pr (x < \max_{1 \leq j \leq p} Y_{j} \leq x+e_{p,\beta}+\delta ) \leq \LL ( \max_{1 \leq j \leq p} Y_{j}, e_{p,\beta}+\delta) \label{eq: Kdistance-intermediate2} \\
&\leq  C (e_{p,\beta} + \delta) \sqrt{1 \vee a_{p}^{2} \vee \log \{ 1 / (e_{p,\beta} + \delta) \}} \leq  C (e_{p,\beta} + \delta) \sqrt{1 \vee a_{p}^{2} \vee  \log (1 / \delta)}. \notag
\end{align}
Therefore,
\begin{multline}
\Pr ( \max_{1\leq j\leq p}X_{j}\leq x ) -  \Pr (\max_{1\leq j\leq p}Y_{j}\leq x)  \\
\leq C  \left \{ (\delta^{-2}+\beta \delta^{-1})\Delta +  (e_{p,\beta}+\delta)\sqrt{1 \vee a_{p}^{2} \vee \log(1/\delta)}  \right \}. \label{normalcomp}
\end{multline}
Take $a = a_{p} \vee \log^{1/2}(1/\Delta)$, and choose  $\beta$ and $\delta$ in such a way that
\begin{equation*}
\beta = \delta^{-1}\log p \ \text{and} \ \delta = \Delta^{1/3} (1 \vee a)^{-1/3} (2\log p)^{1/3}.
\end{equation*}
Recall that $p \geq 2$ and $0 < \Delta < 1$. Observe that $(\delta^{-2}+\beta \delta^{-1})\Delta \leq C \Delta^{1/3} (1 \vee a)^{2/3} \log^{1/3} p, (e_{p,\beta}+\delta) (1 \vee a_{p}) \leq C \Delta^{1/3} (1 \vee a)^{2/3} \log^{1/3} p$, and since $\delta \geq \Delta^{1/3} (1 \vee a)^{-1/3}$, we have  $\log (1/\delta) \leq (1/3) \log ((1 \vee a)/\Delta)$. Hence the right side on (\ref{normalcomp}) is bounded by $C \Delta^{1/3} \{ (1 \vee a)^{2/3} \log^{1/3} p + (1 \vee a)^{-1/3} ( \log^{1/3} p ) \log^{1/2}((1 \vee a)/\Delta) \}$. In addition, $(1 \vee a)^{-1} \log^{1/2} (1 \vee a)$ is bounded by a universal constant, so that
\begin{multline*}
\text{right side of (\ref{normalcomp})} \\
\leq C \Delta^{1/3} \{ (1 \vee a)^{2/3} \log^{1/3} p + (1 \vee a)^{-1/3} ( \log^{1/3} p ) \log^{1/2}(1/\Delta) \}.
\end{multline*}
The second term inside the bracket is bounded by $( \log^{1/3} p ) \log^{1/3}(1/\Delta)$ as $(1 \vee a)^{-1/3} \leq (\log (1/\Delta))^{-1/6}$,
and first term is bound by $(1 \vee a_{p})^{2/3} \log^{1/3} p + ( \log^{1/3} p ) \log^{1/3}(1/\Delta)$. Adjusting the constant $C$, the right side on the above displayed equation is bounded by $C \Delta^{1/3} \{ 1 \vee a_{p}^{2} \vee \log(1/\Delta) \}^{1/3} \log^{1/3}p$.

For the opposite direction, observe that
\begin{align*}
\Pr ( \max_{1 \leq j \leq p} X_{j}   \leq x )
&\geq \Pr ( F_{\beta} (X) \leq x) \geq \Ep[g_{x-e_{p,\beta}-\delta,\beta,\delta}(F_\beta(X))] \\
&\geq \Ep[g_{x-e_{p,\beta}-\delta,\beta,\delta}(F_\beta(Y))] - C(\delta^{-2}+\beta \delta^{-1})\Delta \\
&\geq \Pr (F_{\beta}(Y) \leq x - \delta ) - C(\delta^{-2}+\beta \delta^{-1})\Delta \\
&\geq \Pr (\max_{1 \leq j \leq p} Y_{j} \leq x - e_{p,\beta} -  \delta ) - C(\delta^{-2}+\beta \delta^{-1})\Delta.
\end{align*}
The rest of the proof is similar and hence omitted.
\qed

\section{Proof of Theorem \ref{thm: anticoncentration}}
\label{sec: proof for section 3}

The proof of Theorem  \ref{thm: anticoncentration} uses some properties of Gaussian measures.
We begin with preparing technical tools. The following two facts were essentially noted in \cite{Y65,Y68} (note: \cite{Y65} and \cite{Y68} did not contain a proof of Lemma \ref{lemma: Y65-1}, which we find non-trivial). For the sake of completeness, we give their proofs after the proof of Theorem \ref{thm: anticoncentration}.
\begin{lemma}
\label{lemma: Y65-1}
Let $(W_{1},\dots,W_{p})^{T}$ be a (not necessarily centered) Gaussian random vector in $\RR^{p}$ with $\Var(W_{j}) = 1$ for all $1 \leq j \leq p$.  Suppose that $\Corr(W_{j},W_{k})<1$ whenever $j \neq k$. Then the distribution of $\max_{1 \leq j \leq p} W_{j}$ is absolutely continuous with respect to the Lebesgue measure and a version of the density is given by
\begin{equation}
f(x) = \phi(x) \sum_{j=1}^{p} e^{ \Ep[W_{j}] x-(\Ep[W_{j}])^2/2 } \cdot \Pr\left (W_{k} \leq x, \forall k \neq j \mid W_{j} = x \right). \label{density}
\end{equation}
\end{lemma}

\begin{lemma}
\label{lemma: Y65-2}
Let $(W_{0},W_{1},\dots,W_{p})^{T}$ be  a (not necessarily centered) Gaussian random vector with $\Var(W_{j}) =1 $ for all $0 \leq j \leq p$. Suppose that $\Ep[W_{0} ]\geq 0$. Then the map
\begin{equation}\label{eq: prob}
x \mapsto  e^{ \Ep[W_{0} ] x-(\Ep[W_{0}])^2/2 } \cdot\Pr( W_{j} \leq x, 1 \leq \forall j \leq p  \mid W_{0} = x)
\end{equation}
is non-decreasing on $\RR$.
\end{lemma}

Let us also recall (a version of) the Gaussian concentration (more precisely, deviation) inequality. See,  for example,  \cite{L01}, Theorem 7.1, for its proof.
\begin{lemma}\label{thm3}
Let $(X_{1},\dots,X_{p})^{T}$ be a centered Gaussian random vector in $\RR^{p}$ with $\max_{1 \leq j \leq p} \Ep [X_{j}^{2}] \leq  \sigma^{2}$ for some $\sigma^{2}> 0$. Then for every $r > 0$,
\begin{equation*}
\Pr  ( \max_{1 \leq j \leq p} X_{j} \geq \Ep [ \max_{ 1 \leq j \leq p} X_{j}  ] + r ) \leq e^{-r^{2}/(2\sigma^{2})}.
\end{equation*}
\end{lemma}

We are now in position to prove Theorem \ref{thm: anticoncentration}.

\begin{proof}[Proof of Theorem \ref{thm: anticoncentration}]
The proof consists of three steps.

{\bf Step 1}. This step reduces the analysis to the unit variance case. Pick any $x\geq0$. Let $W_{j} :=(X_{j}-x)/\sigma_{j}+x/\underline{\sigma}$.
Then $\Ep[W_{j}]\geq 0$ and $\Var (W_{j})=1$. Define $Z:= \max_{1 \leq j \leq p} W_{j}$.
Then we have
\begin{align*}
&\Pr (  |\max_{1 \leq j \leq p}X_{j}-x |\leq \epsilon ) \leq \Pr\left(\left|\max_{1 \leq j \leq p}\frac{X_{j}-x}{\sigma_j}\right|\leq \frac{\epsilon}{\underline{\sigma}}\right) \\
&\quad \leq \sup_{y\in\RR}\Pr\left(\left|\max_{1 \leq j \leq p}\frac{X_{j}-x}{\sigma_j}+\frac{x}{\underline{\sigma}}-y\right|\leq \frac{\epsilon}{\underline{\sigma}}\right) =\sup_{y\in\RR}\Pr\left(\left|Z-y\right|\leq \frac{\epsilon}{\underline{\sigma}}\right).
\end{align*}

{\bf Step 2}. This step bounds the density of $Z$.
Without loss of generality, we may assume  that $\Corr (W_{j},W_{k})<1$ whenever $j \neq k$. Since the marginal distribution of $W_{j}$ is $N(\mu_{j},1)$ where $\mu_{j}:=\Ep[W_{j}]=(x/\underline{\sigma}-x/\sigma_j)\geq 0$, by Lemma \ref{lemma: Y65-1}, $Z$ has density of the form
\begin{equation}\label{eq: density form}
f_{p}(z) = \phi(z) G_{p} (z),
\end{equation}
where the map $z \mapsto G_{p}(z)$ is non-decreasing by Lemma \ref{lemma: Y65-2}. Define $\bar{z}:= (1/\underline{\sigma}-1/\overline{\sigma})x$, so that $\mu_{j}\leq \bar{z}$ for every $1 \leq j \leq p$. Moreover, define $\bar{Z}:=\max_{1 \leq j \leq p}(W_{j}-\mu_{j})$.  Then
\begin{align*}
&\int_{z}^{\infty} \phi (u) du G_{p}(z) \leq \int_{z}^{\infty} \phi(u) G_{p}(u) du = \Pr( Z> z) \\
&\qquad \leq \Pr( \bar{Z} > z - \bar{z} ) \leq  \exp \left \{ -\frac{(z-\bar{z}-\Ep[ \bar{Z}  ])_{+}^{2}}{2} \right \},
\end{align*}
where the last inequality is due to the Gaussian concentration inequality (Lemma \ref{thm3}).
Note that $W_{j}-\mu_{j}=X_{j}/\sigma_{j}$, so that
\begin{equation*}
\Ep[\bar{Z}]=\Ep [ \max_{1 \leq j \leq p}(X_{j}/\sigma_{j}) ] =: a_p.
\end{equation*}
Therefore, for every $z \in \RR$,
\begin{equation}\label{eq: bound Gn}
G_{p} (z) \leq \frac{1}{1-\Phi(z)} \exp \left \{ -\frac{(z-\bar{z}-a_p)_{+}^{2}}{2} \right \}.
\end{equation}
Mill's inequality states that for $z>0$,
\begin{equation*}
z \leq \frac{\phi(z)}{1-\Phi (z)}\leq z \frac{1+z^2}{z^2},
\end{equation*}
and in particular $(1+z^2)/z^2 \leq 2$ when $z>1$. Moreover, $\phi(z)/\{ 1-\Phi(z) \}\leq 1.53 \leq 2 $ for  $z \in (-\infty, 1)$.  Therefore,
\begin{equation*}
\phi(z)/\{ 1-\Phi(z) \} \leq 2 (z \vee 1),  \ \forall z \in \RR.
\end{equation*}
Hence we conclude from this, (\ref{eq: bound Gn}), and (\ref{eq: density form}) that
\begin{equation*}
f_{p}(z) \leq 2(z \vee 1) \exp \left \{  -\frac{(z-\bar{z}-a_p)_{+}^{2}}{2} \right \}, \ \forall z \in \RR.
\end{equation*}

{\bf Step 3}. By Step 2, for every $y \in \RR$ and $t > 0$,  we have
\begin{align*}
\Pr\left( | Z - y | \leq t \right) &= \int_{y -t}^{y+t} f_{p}(z) dz
\leq 2 t \max_{z \in [y -t, y+t]}f_{p}(z)  \leq 4 t (\bar{z}+ a_p + 1),
\end{align*}
where the last inequality follows from the fact that the map $z \mapsto z e^{-(z-a)^{2}/2}$ (with $a > 0$) is non-increasing on $[a+1,\infty)$.
Combining this inequality with  Step 1, for every $x \geq 0$ and $\epsilon > 0$, we have
\begin{equation}
\Pr ( | \max_{1 \leq j \leq p}X_{j} - x | \leq \epsilon ) \leq 4 \epsilon \{  (1/\underline{\sigma}-1/\overline{\sigma}) |x|+a_p + 1 \} /\underline{\sigma}.  \label{eq: anticoncentration main}
\end{equation}
This inequality also holds for $x<0$ by the similar argument, and hence it holds
for every $x \in \RR$.

If  $\underline{\sigma}=\overline{\sigma} =\sigma$, then we have
\begin{equation*}
\Pr ( | \max_{1 \leq j \leq p}X_{j} - x | \leq \epsilon ) \leq 4 \epsilon (a_p + 1)/\sigma, \ \forall  x \in \RR, \ \forall \epsilon > 0,
\end{equation*}
which leads to the first assertion of the theorem.

On the other hand, consider the case where $\underline{\sigma} < \overline{\sigma}$. Suppose first that $0 < \epsilon \leq \underline{\sigma}$.
By the Gaussian concentration inequality (Lemma \ref{thm3}), for $|x|\geq \epsilon+\overline \sigma(a_p+\sqrt{2\log(\underline{\sigma}/\epsilon)})$, we have
\begin{align}
&\Pr ( | \max_{1 \leq j \leq p} X_{j}-x|\leq\epsilon ) \leq \Pr (\max_{1 \leq j \leq p}X_{j} \geq |x|-\epsilon ) \notag \\
&\quad \leq \Pr \left (\max_{1 \leq j \leq p}X_{j}\geq \Ep[\max_{1 \leq j \leq p}X_{j}] +\overline{\sigma}\sqrt{2\log(\underline{\sigma}/\epsilon)} \right ) \leq \epsilon/\underline{\sigma}. \label{eq: anticoncentration tail}
\end{align}
For  $|x| \leq \epsilon+\overline \sigma(a_p+\sqrt{2\log(\underline{\sigma}/\epsilon)})$, by  (\ref{eq: anticoncentration main}) and using $\epsilon \leq \underline{\sigma}$, we have
\begin{align}
&\Pr ( | \max_{1 \leq j \leq p}X_{j} - x | \leq \epsilon ) \notag \\
&\quad \leq 4 \epsilon \{ (\overline{\sigma}/\underline{\sigma}) a_{p} +  (\overline{\sigma}/\underline{\sigma}-1)\sqrt{2\log(\underline{\sigma}/\epsilon)}   + 2 - \underline{\sigma}/\overline{\sigma}\}/\underline{\sigma}.
 \label{eq: anticoncentration non-tail}
\end{align}
Combining (\ref{eq: anticoncentration tail}) and (\ref{eq: anticoncentration non-tail}), we obtain the inequality in (ii) for $0 < \epsilon \leq \underline{\sigma}$ (with a suitable choice of $C$). If $\epsilon> \underline{\sigma}$, the inequality in (ii) trivially follows by taking $C \geq 1/\underline{\sigma}$.  This completes the proof.
\end{proof}

\begin{proof}[Proof of Lemma \ref{lemma: Y65-1}]
Let $M := \max_{1 \leq j \leq p} W_{j}$.
The absolute continuity of the distribution of $M$ is deduced from the fact that $\Pr(M \in A) \leq \sum_{j=1}^{p} \Pr(W_{j} \in A)$ for every Borel measurable subset $A$ of $\RR$.
Hence, to show that a version of the density of $M$ is given by (\ref{density}), it is enough to show that $\lim_{\epsilon \downarrow 0} \epsilon^{-1} \Pr(x < M \leq x+\epsilon)$ equals the right side on (\ref{density}) for a.e. $x \in \RR$.

For every $x \in \RR$ and $\epsilon > 0$, observe that
\begin{align*}
& \{ x < M \leq x + \epsilon \} \\
& = \{ \exists i_{0}, W_{i_{0}} > x \ \text{and} \ \forall i,  W_{i} \leq x +\epsilon  \} \\
&= \{ \exists i_{1}, x < W_{i_{1}} \leq x + \epsilon \ \text{and} \ \forall i \neq i_{1}, W_{i} \leq x \} \\
&\quad \cup \{ \exists i_{1}, \exists i_{2}, x < W_{i_{1}} \leq x + \epsilon, x < W_{i_{2}} \leq x + \epsilon \ \text{and} \ \forall i \notin \{ i_{1},i_{2} \},   W_{i} \leq x \} \\
&\qquad \qquad \vdots \\
&\quad \cup \{ \forall i, x < W_{i} \leq x + \epsilon \} \\
&=: A^{x,\epsilon}_{1} \cup A^{x,\epsilon}_{2} \cup \cdots \cup A^{x,\epsilon}_{p}.
\end{align*}
Note that the events $A^{x,\epsilon}_{1},A^{x,\epsilon}_{2},\dots,A^{x,\epsilon}_{p}$ are disjoint. For $A^{x,\epsilon}_{1}$, since
\begin{equation*}
A^{x,\epsilon}_{1} = \cup_{i=1}^{p} \{ x < W_{i} \leq x+\epsilon \ \text{and} \ W_{j} \leq x, \forall j \neq i \},
\end{equation*}
where the events on the right side are disjoint, we have
\begin{align*}
\Pr(A^{x,\epsilon}_{1}) &= \sum_{i=1}^{p} \Pr(x < W_{i} \leq x+\epsilon \ \text{and} \ W_{j} \leq x, \forall j \neq i) \\
&= \sum_{i=1}^{p} \int_{x}^{x+\epsilon} \Pr(W_{j} \leq x, \forall j \neq i \mid W_{i} = u) \phi(u-\mu_{i}) du,
\end{align*}
where $\mu_{i} := \Ep [ W_{i} ]$.
We show that for every $1 \leq i \leq p$ and a.e. $x \in \RR$, the map $u \mapsto \Pr(W_{j} \leq x, \forall j \neq i \mid W_{i} = u)$ is right continuous at $x$.
Let $X_{j}=W_{j}-\mu_{j}$ so that $X_{j}$ are standard Gaussian random variables. Then
\begin{equation*}
\Pr(W_{j} \leq x, \forall j \neq i \mid W_{i} = u)=\Pr(X_{j} \leq x-\mu_{j}, \forall j \neq i \mid X_{i} = u-\mu_{i}).
\end{equation*}
Pick $i=1$. Let $V_{j} = X_{j} - \Ep[ X_{j} X_{1}] X_{1}$ be the residual from the orthogonal projection of $X_{j}$ on $X_{1}$.
Note that the vector $(V_{j})_{2 \leq j \leq p}$ and $X_{1}$ are jointly Gaussian and uncorrelated, and hence independent, by which we have
\begin{align*}
&\Pr(X_{j} \leq x-\mu_{j}, 2 \leq \forall j \leq p \mid X_{1} = u-\mu_{1}) \\
&= \Pr(V_{j} \leq x-\mu_{j}-\Ep[ X_{j} X_{1}] (u-\mu_{1}), 2 \leq \forall j \leq p \mid X_{1} = u-\mu_{1}) \\
&=  \Pr(V_{j} \leq x-\mu_{j}-\Ep[ X_{j} X_{1}] (u-\mu_{1}), 2 \leq \forall j \leq p).
\end{align*}
Define $J := \{ j \in \{ 2,\dots,p \} : \Ep[ X_{j} X_{1}] \leq 0 \}$ and $J^{c} := \{ 2,\dots, p\} \backslash J$. Then
\begin{align*}
&\Pr(V_{j} \leq x-\mu_{j}-\Ep[ X_{j} X_{1}] (u-\mu_{1}), 2 \leq \forall j \leq p) \\
&\to \Pr( V_{j} \leq x_j, \forall j \in J, V_{j'} < x_{j'}, \forall j' \in J^{c}), \ \text{as} \ u \downarrow x,
\end{align*}
where $x_j=x-\mu_{j}-\Ep[X_{j}X_1](x-\mu_{1})$.
Here each $V_{j}$ either degenerates to $0$ (which occurs only when $X_{j}$ and $X_{1}$ are perfectly negatively correlated, that is, $\Ep [ X_{j} X_{1} ] = -1$) or has a non-degenerate Gaussian distribution, and hence for every $x \in \RR$ expect for at most $(p-1)$ points $(\mu_{1}+\mu_{j})/2, 2 \leq j \leq p$,
\begin{align*}
\Pr( V_{j} \leq x_j, \forall j \in J, V_{j'} < x_{j'}, \forall j' \in J^{c}) &=  \Pr( V_{j} \leq x_{j}, 2 \leq \forall j \leq p ) \\
&= \Pr(W_{j} \leq x, 2 \leq \forall j \leq p \mid W_{1} = x).
\end{align*}
Hence for $i=1$ and a.e. $x \in \RR$, the map $u \mapsto \Pr(W_{i} \leq x, \forall j \neq i \mid W_{i} = u)$ is right continuous at $x$. The same conclusion clearly holds for $2 \leq i \leq p$.
Therefore, we conclude that, for a.e. $x \in \RR$, as $\epsilon \downarrow 0$,
\begin{eqnarray*}
\frac{1}{\epsilon} \Pr(A_{1}^{x,\epsilon}) &\to& \sum_{i=1}^{p} \Pr(W_{j} \leq x, \forall j \neq i \mid W_{i} = x)\phi(x-\mu_{i})\\
&=&\phi(x)\sum_{i=1}^{p} e^{\mu_{i}x-\mu_{i}^2/2} \Pr(W_{j} \leq x, \forall j \neq i \mid W_{i} = x).
\end{eqnarray*}

In the rest of the proof, we  show that, for every $2 \leq i \leq p$ and $x \in \RR$, $\Pr(A_{i}^{x,\epsilon}) = o(\epsilon)$ as $\epsilon \downarrow 0$, which leads to the desired conclusion. Fix any $2 \leq i \leq p$. The probability $\Pr(A_{i}^{x,\epsilon})$ is bounded by a sum of terms of the form  $\Pr(x < W_{j} \leq x+\epsilon, x < W_{k} \leq x+\epsilon)$ with $j \neq k$. Recall that $\Corr(W_{j},W_{k})<1$. Assume that $\Corr(W_{j},W_{k})=-1$. Then for  every $x \in \RR$, $\Pr(x < W_{j} \leq x+\epsilon, x < W_{k} \leq x+\epsilon)$ is zero for sufficiently small $\epsilon$. Otherwise, $(W_{j},W_{k})^{T}$ obeys a two-dimensional, non-degenerate Gaussian distribution and hence $\Pr(x < W_{j} \leq x+\epsilon, x < W_{k} \leq x+\epsilon) = O(\epsilon^{2}) = o(\epsilon)$ as $\epsilon \downarrow 0$ for every $x \in \RR$. This completes the proof.
\end{proof}

\begin{proof}[Proof of Lemma \ref{lemma: Y65-2}]
Since $\Ep [ W_{0} ] \geq 0$, the map $x \mapsto \exp( \Ep[ W_{0} ] x- (\Ep[ W_{0} ])^{2} )$ is non-decreasing. Thus it suffices to show that the map
\begin{equation}\label{eq: reduced probability}
x \mapsto \Pr ( W_{1} \leq x, \dots,  W_{p} \leq x \mid W_{0} =x)
\end{equation}
is non-decreasing.
As in the proof of Lemma \ref{lemma: Y65-1}, let $X_{j}=W_{j}-\Ep [ W_{j} ]$ and let $V_{j} = X_{j} - \Ep[ X_{j} X_{0}] X_{0}$ be the residual from the orthogonal projection of $X_{j}$ on $X_{0}$.  Note that the vector $(V_{j})_{1 \leq j \leq p}$ and $X_{0}$ are independent.
Hence the  probability in (\ref{eq: reduced probability}) equals
\begin{align*}
&\Pr(V_{j} \leq x-\mu_{j} - \Ep[X_{j}X_{0}](x-\Ep[ W_{0} ] ),  1  \leq \forall j \leq p \mid X_{0}=x-\Ep[ W_{0} ])\\
&=\Pr(V_{j} \leq x-\mu_{j} - \Ep[X_{j}X_{0}](x-\Ep[ W_{0} ]),  1  \leq \forall j \leq p),
\end{align*}
where the latter is non-decreasing in $x$ on $\RR$ since $\Ep[X_{j}X_{0}] \leq 1$.
\end{proof}

\appendix

\section{Proof of Lemma \ref{lem: Delta estimate}}

Lemma \ref{lem: Delta estimate} follows from the following maximal inequality and H\"{o}lder's inequality. Here we write $a \lesssim b$  if $a$ is smaller than or equal to $b$ up to a universal positive constant.

\begin{lemma}
\label{lem: maximal ineq}
Let $Z_{1},\dots,Z_{n}$ be independent random vectors in $\RR^{p}$ with $p \geq 2$. Define $M := \max_{1 \leq i \leq n} \max_{1 \leq j \leq p} | Z_{ij} |$ and $\sigma^{2} := \max_{1 \leq j \leq p} \sum_{i=1}^{n} \Ep [ Z_{ij}^{2} ]$.
Then
\begin{equation*}
\Ep [\max_{1 \leq j \leq p} | {\textstyle \sum}_{i=1}^{n} (Z_{ij} - \Ep [Z_{ij}]) | ] \lesssim (\sigma \sqrt{\log p} + \sqrt{\Ep [ M^{2} ]} \log p).
\end{equation*}
\end{lemma}

We shall use the following lemma.

\begin{lemma}
\label{lem: maximal ineq nonnegative}
Let $V_{1},\dots,V_{n}$ be independent random vectors in $\RR^{p}$ with $p \geq 2$ such that $V_{ij} \geq 0$ for all $1 \leq i \leq n$ and $1 \leq j \leq p$. Then
\begin{equation*}
\Ep [ \max_{1 \leq j \leq p} {\textstyle \sum}_{i=1}^{n} V_{ij} ] \lesssim \max_{1 \leq j \leq p} \Ep [ {\textstyle \sum}_{i=1}^{n} V_{ij} ] + \Ep [ \max_{1 \leq i \leq n} \max_{1 \leq j \leq p} V_{ij} ] \log p.
\end{equation*}
\end{lemma}

\begin{proof}[Proof of Lemma \ref{lem: maximal ineq nonnegative}]
We make use of the symmetrization technique. Let $\eps_{1},\dots,\eps_{n}$ be independent Rademacher random variables (that is, $\Pr (\eps_{i} = 1) = \Pr ( \eps_{i} = -1) = 1/2$) independent of $V_{1}^{n} := \{ V_{1},\dots,V_{n} \}$.
Then by the triangle inequality and  Lemma 2.3.1 in \cite{VW96},
\begin{align*}
I := \Ep [ \max_{1 \leq j \leq p} {\textstyle \sum}_{i=1}^{n} V_{ij} ] &\leq \max_{1 \leq j \leq p} \Ep [ {\textstyle \sum}_{i=1}^{n} V_{ij} ] + \Ep [ \max_{1 \leq j \leq p} | {\textstyle \sum}_{i=1}^{n} (V_{ij} - \Ep [ V_{ij} ]) |] \\
&\leq \max_{1 \leq j \leq p} \Ep [ {\textstyle \sum}_{i=1}^{n} V_{ij} ] + 2\Ep [ \max_{1 \leq j \leq p} | {\textstyle \sum}_{i=1}^{n} \eps_{i} V_{ij}  |].
\end{align*}
By Lemmas 2.2.2 and 2.2.7 in \cite{VW96}, we have
\begin{align*}
\Ep [ \max_{1 \leq j \leq p} | {\textstyle \sum}_{i=1}^{n} \eps_{i} V_{ij}  | \mid V_{1}^{n} ] &\lesssim \max_{1 \leq j \leq p} ({\textstyle \sum}_{i=1}^{n} V_{ij}^{2})^{1/2}  \sqrt{\log p}\\
&\leq   \sqrt{B \log p} \max_{1 \leq j \leq p} ({\textstyle \sum}_{i=1}^{n} V_{ij})^{1/2},
\end{align*}
where $B := \max_{1 \leq i \leq n} \max_{1 \leq j \leq p} V_{ij}$. Hence by Fubini's theorem and the Cauchy-Schwarz inequality,
\begin{align*}
\Ep [ \max_{1 \leq j \leq p} | {\textstyle \sum}_{i=1}^{n} \eps_{i} V_{ij}  | ] &\lesssim \sqrt{\Ep [ B ] \log p} ( \Ep [ \max_{1 \leq j \leq p} {\textstyle \sum}_{i=1}^{n} V_{ij} ])^{1/2} \\
& = \sqrt{\Ep [ B ] \log p}  \sqrt{I}.
\end{align*}
Therefore, we have
\begin{equation*}
I \lesssim \max_{1 \leq j \leq p} \Ep [ {\textstyle \sum}_{i=1}^{n} V_{ij} ] + \sqrt{\Ep [ B ] \log p} \sqrt{I} =: a + b \sqrt{I}.
\end{equation*}
Solving this inequality, we conclude that $I \lesssim a + b^{2}$.
\end{proof}

\begin{proof}[Proof of Lemma \ref{lem: maximal ineq}]
Let $\eps_{1},\dots,\eps_{n}$ be independent Rademacher random variables independent of $Z_{1},\dots,Z_{n}$. Then arguing as in the previous proof, we have
\begin{align*}
\Ep [\max_{1 \leq j \leq p} | {\textstyle \sum}_{i=1}^{n} (Z_{ij} - \Ep [Z_{ij}]) | ]
&\leq 2 \Ep [\max_{1 \leq j \leq p} | {\textstyle \sum}_{i=1}^{n} \eps_{i}Z_{ij} | ] \\
&\lesssim \Ep [ \max_{1 \leq j \leq p} ( {\textstyle \sum}_{i=1}^{n} Z_{ij}^{2} )^{1/2} ]  \sqrt{\log p} \\
&\leq (\Ep [ \max_{1 \leq j \leq p} {\textstyle \sum}_{i=1}^{n} Z_{ij}^{2}  ] )^{1/2} \sqrt{\log p}. \quad (\text{Jensen})
\end{align*}
By Lemma \ref{lem: maximal ineq nonnegative} applied to $V_{ij} = Z_{ij}^{2}$, we have
\begin{equation*}
\Ep [ \max_{1 \leq j \leq p} {\textstyle \sum}_{i=1}^{n} Z_{ij}^{2}  ] \lesssim \sigma^{2} + \Ep [ M^{2} ] \log p.
\end{equation*}
This implies the desired conclusion.
\end{proof}

\section*{Acknowledgments}
V. Chernozhukov and D. Chetverikov  are supported by a National Science Foundation grant. K. Kato is supported by the Grant-in-Aid for Young Scientists (B) (22730179, 25780152), the Japan Society for the Promotion of Science. We would like to thank the Editors and an anonymous referee for their careful review.

\end{document}